\documentclass[12pt,reqno]{amsart}
\numberwithin{equation}{section}

\newtheorem{theorem}{Theorem}[section]
\newtheorem{corollary}[theorem]{Corollary}
\newtheorem{lemma}[theorem]{Lemma}
\newtheorem{proposition}[theorem]{Proposition}

\theoremstyle{definition}
\newtheorem{definition}[theorem]{Definition}

\theoremstyle{remark}
\newtheorem{remark}[theorem]{Remark}

\begin{document}

\title[Higgs bundles and representation spaces associated to morphisms]{Higgs bundles
and representation spaces associated to morphisms}

\author[I. Biswas]{Indranil Biswas}

\address{School of Mathematics, Tata Institute of Fundamental
Research, Homi Bhabha Road, Bombay 400005, India}

\email{indranil@math.tifr.res.in}

\author[C. Florentino]{Carlos Florentino}

\address{Departamento Matem\'atica, IST, University of Lisbon, Av. Rovisco
Pais, 1049-001 Lisbon, Portugal}

\email{carlos.florentino@tecnico.ulisboa.pt}

\subjclass[2000]{14J60}

\keywords{Higgs bundle, flat connection, representation space, deformation retraction.}

\date{}

\begin{abstract}
Let $G$ be a connected reductive affine algebraic group defined over the
complex numbers, and
$K\,\subset\, G$ be a maximal compact subgroup. Let $X\, , Y$ be irreducible smooth complex
projective varieties and $f\,:\, X\, \longrightarrow\, Y$ an
algebraic morphism, such that $\pi_1(Y)$ is virtually nilpotent and the
homomorphism $f_*\, :\, \pi_1(X)\, \longrightarrow\,\pi_1(Y)$ is surjective.
Define
$$
{\mathcal R }^f(\pi_1(X),\, G)\,=\, \{\rho\, \in\, \text{Hom}(\pi_1(X),\, G)\,
\mid\, A\circ\rho \ \text{ factors through }~ f_*\}\, ,
$$
$$
{\mathcal R }^f(\pi_1(X),\, K)\,=\, \{\rho\, \in\, \text{Hom}(\pi_1(X),\, K)\,
\mid\, A\circ\rho \ \text{ factors through }~ f_*\}\, ,
$$
where $A\,:\, G\, \longrightarrow\, \text{GL}(\text{Lie}(G))$ is the adjoint action.
We prove that the geometric invariant theoretic quotient
${\mathcal R }^f(\pi_1(X, x_0),\, G)/\!\!/G$ admits a deformation
retraction to ${\mathcal R }^f(\pi_1(X, x_0),\, K)/K$. We also show that the space of
conjugacy classes of $n$ almost commuting elements in $G$ admits a deformation
retraction to the space of conjugacy classes of $n$ almost commuting elements in $K$.
\end{abstract}

\thanks{The first author is supported by a J. C. Bose Fellowship. The second author is
partially supported by FCT (Portugal) through
the projects EXCL/MAT-GEO/0222/2012, PTDC/MAT/120411/2010 and PTDC/MAT-GEO/0675/2012.}

\maketitle

\section{Introduction}

Let $G$ be a connected reductive affine algebraic group defined over
the complex numbers. Consider an algebraic morphism
$$
f\, :\, X\, \longrightarrow\, Y
$$
where $X$ and $Y$ are irreducible smooth complex projective varieties, and let
$$
f_*\, :\, \pi_1(X, x_0)\, \longrightarrow\, \pi_1(Y, f(x_0))
$$
be the induced morphism of fundamental groups, where $x_0\,\in\,X$ is a base point.
In certain situations, the representations 
$$
\rho: \pi_1(X, x_0) \, \longrightarrow\, G
$$
that factor through $f_*$ have special geometric properties. 
See \cite{KP}, where necessary and sufficient conditions
for such a factorization are given in terms of the spectral curve 
of the $G$-Higgs bundle associated to $\rho$.

In this article, we are interested in the whole moduli space of representations 
that factor in a similar way, and in its topological properties.
Under some assumptions on $f$ and $Y$, we provide a
natural deformation retraction between two such representation spaces, 
described as follows.

The Lie algebra of $G$ will
be denoted by $\mathfrak g$. 
Let $A\,:\, G\, \longrightarrow\, \text{GL}({\mathfrak g})$ be the homomorphism
given by the adjoint action of $G$ on $\mathfrak g$. Fix a maximal compact
subgroup $K\, \subset\, G$ and define:
$$
{\mathcal R }^f(\pi_1(X, x_0),\, G)\,=\, \{\rho\, \in\, \text{Hom}(\pi_1(X,x_0),\, G)\,
\mid\, A\circ\rho \ \text{ factors through }~ f_*\}\, ,
$$
$$
{\mathcal R }^f(\pi_1(X, x_0),\, K)\,=\, \{\rho\, \in\, \text{Hom}(\pi_1(X,x_0),\, K)\,
\mid\, A\circ\rho \ \text{ factors through }~ f_*\}\, .
$$
We note that the group $G$ (respectively, $K$) acts on ${\mathcal R }^f(\pi_1(X, x_0),\, G)$
(respectively, on ${\mathcal R }^f(\pi_1(X, x_0),\, K)$) via the conjugation action
of $G$ (respectively, $K$) on itself.
The quotient ${\mathcal R }^f(\pi_1(X, x_0),\, K)/K$
is contained in the geometric invariant theoretic quotient
${\mathcal R }^f(\pi_1(X, x_0),\, G)/\!\!/G$.

We prove the following in Theorem \ref{thm1}:

\textit{ Suppose that the fundamental group of $Y$ is virtually nilpotent, and the homomorphism $f_*$ is surjective. 
Then ${\mathcal R }^f(\pi_1(X, x_0),\, G)/\!\!/G$ admits a deformation
retraction to the subset ${\mathcal R }^f(\pi_1(X, x_0),\, K)/K$.}

In Section 3, we consider spaces of almost commuting elements in $K$ and in $G$. 
Define:
$$
{\rm AC}^n(K)\,=\, \{(g_1\, ,\cdots\, , g_n)\,\in\, K^n\, \mid\, g_ig_jg^{-1}_ig^{-1}_j
\,\in\, Z_K~\, \ \forall \ i\, ,j\}\, ,
$$
where $Z_K$ denotes the center of $K$. The moduli space of conjugacy classes:
$$
{\rm AC}^n(K) \, / \, K\, ,
$$
where $K$ acts by simultaneous conjugation, was studied in \cite{BFM}, \cite{KS},
and plenty of information is known in the cases $n=2$ and $n=3$. 
For instance, the number of components of ${\rm AC}^3(K) \, / \, K$ 
has been related in \cite{BFM} to the Chern--Simons invariants 
associated to flat connections on a 3-torus.

In a similar fashion, we define ${\rm AC}^n(G)/\!\!/G$, the moduli space of
conjugacy classes of $n$ almost commuting elements in $G$.
For example, if $G$ has trivial center, then ${\rm AC}^{2n}(G)/\!\!/G$ coincides with
$$
\text{Hom}(\pi_1(X,x_0),\, G)/\!\!/G\, ,
$$ 
where $X$ is an abelian variety of complex dimension $n$.
In Proposition \ref{prop1}, we show that ${\rm AC}^n(G) \, / \, G$ admits a
deformation retraction to ${\rm AC}^n(K) \, / \, K$, and that the same holds for
${\rm AC}^n(G)$ and ${\rm AC}^n(K)$, extending one of the main results in
\cite{FL} and \cite{BF1}.

\section{Representation spaces associated to a morphism}

Let $X$ be an irreducible smooth complex projective variety. Fix a point
$x_0\, \in\, X$. Let
\begin{equation*}
f\, :\, X\, \longrightarrow\, Y
\end{equation*}
be an algebraic morphism, where $Y$ is also an irreducible smooth complex projective
variety, such that:
\begin{enumerate}
\item the fundamental group $\pi_1(Y, f(x_0))$ is virtually nilpotent, and

\item the homomorphism of fundamental groups induced by $f$
\begin{equation}\label{e2}
f_*\, :\, \pi_1(X, x_0)\, \longrightarrow\,\pi_1(Y, f(x_0))
\end{equation}
is surjective.
\end{enumerate}
Using the homomorphism $f_*$ in \eqref{e2}, we will consider $\pi_1(Y, f(x_0))$ as a
quotient of the group $\pi_1(X,x_0)$.

Let $G$ be a connected reductive affine algebraic group defined over $\mathbb C$. The
Lie algebra of $G$ will be denoted by $\mathfrak g$. Let
\begin{equation}\label{e0}
A\, :\, G\, \longrightarrow\, \text{GL}(\mathfrak g)
\end{equation}
be the homomorphism given by the adjoint action of $G$ on $\mathfrak g$.
The affine algebraic variety (not necessarily irreducible) of representations
$$
\rho\, : \pi_1(X, x_0) \longrightarrow G
$$
will be denoted by $\text{Hom}(\pi_1(X,x_0),\, G)$.

\begin{definition}
Let $\rho \in \text{Hom}(\pi_1(X,x_0),\, G)$. We sat that 
$A\circ \rho$ factors through $f_*$ in \eqref{e2} 
(or that $A\circ \rho$ \emph{factors geometrically} through $f:X\to Y$, see \cite{KP})
if there exists a homomorphism $\rho' \in \text{Hom}(\pi_1(Y,f(x_0)), \text{GL}(\mathfrak g))$
such that 
\begin{equation}\label{e3}
\rho'\circ f_*\,=\, A\circ\rho\, .
\end{equation}
\end{definition}

\begin{remark}
(1) Clearly, if $\rho$ itself factorizes as $\rho = \tilde\rho \, \circ f_*$ for some
$\tilde\rho \in \text{Hom}(\pi_1(X,x_0), G)$, then $A\circ\rho$ factorizes through $f_*$ as in the definition;
the converse is not always true.  \newline
(2) It is clear that $A\circ\rho\in \text{Hom}(\pi_1(X,x_0), \text{GL}(\mathfrak g))$ factors through $f_*$ as in \eqref{e3}, 
if and only if $A\circ\rho$ is trivial on the kernel of $f_*$.
Moreover, when $A\circ\rho$ factors through $f_*$, a homomorphism $\rho' \in \text{Hom}(\pi_1(Y,f(x_0)), \text{GL}(\mathfrak g))$
satisfying equation \eqref{e3} is unique, because $f_*$ is surjective.
\end{remark}

In the framework of non-abelian Hodge theory, there is a correspondence between 
semistable $G$-Higgs bundles over $X$ and representations in 
$\text{Hom}(\pi_1(X,x_0),\, G)$, \cite{Si}, \cite{BG}. Denote by $(E_\rho\, 
,\theta_\rho)$ the semistable $G$--Higgs bundle on $X$ associated to $\rho$ under 
this correspondence. We note that $(E_\rho\, ,\theta_\rho)$ is semistable with 
respect to every polarization on $X$.

\begin{lemma}\label{lem1}
Let $\rho \in \text{Hom}(\pi_1(X,x_0),\, G)$ be such that
$A\circ \rho$ factors through $f_*$. Then, the above principal $G$--bundle $E_\rho$ on $X$ is semistable.
\end{lemma}

\begin{proof}
Let
$$
\text{ad}(E_\rho)\,:=\, E_\rho\times^A \mathfrak g\,\longrightarrow\, X
$$
be the adjoint vector bundle of $E_\rho$. The Higgs field on $\text{ad}(E_\rho)$ induced
by $\theta_\rho$ will be denoted by $\text{ad}(\theta_\rho)$.

Let $\rho'\, :\, \pi_1(Y, f(x_0))\, \longrightarrow\, \text{GL}(\mathfrak g)$ be the
unique homomorphism satisfying equation \eqref{e3}; the uniqueness of $\rho'$ is a consequence
of the surjectivity of $f_*$ as remarked above.
Let $(E'\, ,\theta')$ be the semistable Higgs vector bundle on $Y$ associated to this
homomorphism $\rho'$. Since the fundamental group of $Y$ is virtually nilpotent,
we know that the vector bundle $E'$ is semistable \cite[Proposition 3.1]{BF2}.
Let $c_i(E'), \,  \, i\, \geq 0$, be the sequence of Chern classes of the bundle $E'$. Then,
$c_i(E')\,=\, 0$ for all $i\, >\ 0$ because the $C^\infty$ complex vector bundle
underlying $E'$ admits a flat connection (it is isomorphic to the $C^\infty$ complex vector
bundle underlying the flat vector bundle associated to $\rho'$).
Therefore, by \cite[p. 39, Theorem 5.1]{BB}, the vector bundle $E'$ admits a filtration
$$
0\, =\, V_0 \, \subset\, V_1 \, \subset\, \cdots \, \subset\,V_{\ell-1} \, \subset\, V_\ell
\,=\, E'
$$
of holomorphic subbundles such that each successive quotient $V_i/V_{i-1}$,
$1\,\leq\, i\, \leq\, \ell$, admits a flat unitary connection. Consider the pulled back
filtration
\begin{equation}\label{b1}
0\, =\, f^*V_0 \, \subset\, f^*V_1 \, \subset\, \cdots \, \subset\,f^*V_{\ell-1} \, \subset\,
f^*V_\ell \,=\, f^*E'\, .
\end{equation}
A flat unitary connection on $V_i/V_{i-1}$ pulls back to a flat unitary connection on
$$
f^*V_i/(f^*V_{i-1})\,=\, f^*(V_i/V_{i-1})\, .
$$
Since each successive quotient for the filtration of $f^*E'$ in \eqref{b1} admits a
flat unitary connection, we conclude that the holomorphic vector bundle $f^*E'$ is
semistable.

{}From \eqref{e3} it follows that
\begin{equation}\label{e4}
(\text{ad}(E_\rho)\, , \text{ad}(\theta_\rho))\,=\, (f^*E'\, ,f^*\theta')\, .
\end{equation}
Since $f^*E'$ is semistable, from \eqref{e4} it follows that $\text{ad}(E_\rho)$
is semistable. This implies that the principal $G$--bundle $E_\rho$ is semistable \cite[p. 214,
Proposition 2.10]{AB}.
\end{proof}

Lemma \ref{lem1} has the following corollary:

\begin{corollary}\label{cor1}
For any Higgs field $\theta$, the $G$--Higgs bundle $(E_\rho\, ,\theta)$
is semistable.
\end{corollary}

Let
\begin{equation}\label{e5}
\rho^\lambda\, :\, \pi_1(X, x_0)\, \longrightarrow\, G
\end{equation}
be a homomorphism corresponding to the Higgs $G$--bundle $(E_\rho\, ,\lambda\cdot\theta_\rho)$, which is semistable
by Corollary \ref{cor1}. We note that although $\rho^\lambda$ is not uniquely determined by
$(E_\rho\, ,\lambda\cdot\theta_\rho)$, the point in the quotient space
$$
\text{Hom}(\pi_1(X,x_0),\, G)/G
$$
given by $\rho^\lambda$ does not depend on the choice of $\rho^\lambda$. In other words,
any two different choices of $\rho^\lambda$ differ by an inner automorphism of the group $G$.

\begin{lemma}\label{lem2}
For every $\lambda\, \in\, \mathbb C$, the homomorphism $A\circ\rho^\lambda$ factors through  $f_*$, where 
$\rho^\lambda$ is defined in \eqref{e5}.
\end{lemma}

\begin{proof}
Let $(\text{ad}(E_\rho)^\lambda\, , \text{ad}(\theta_\rho)^\lambda)$ be the Higgs vector bundle
associated to the homomorphism $A\circ\rho^\lambda$.
We note that $(\text{ad}(E_\rho)^\lambda\, , \text{ad}(\theta_\rho)^\lambda)$ is isomorphic to
$(f^*E'\, ,f^*(\lambda\cdot \theta'))$, because the Higgs bundle $(E'\, ,\theta')$
corresponds to $\rho'$, and \eqref{e3} holds. We saw in the proof of Lemma \ref{lem1} that
$E'$ is semistable with $c_i(E')\,=\, 0$ for all $i\, >\, 0$. Since
$(\text{ad}(E_\rho)^\lambda\, , 
\text{ad}(\theta_\rho)^\lambda)$ is isomorphic to the pullback of a semistable Higgs vector 
bundle on $Y$ such that all the Chern classes of positive degrees of the underlying
vector bundle on $Y$ vanish, it can be deduced that $A\circ\rho^\lambda$
factors through the quotient $\pi_1(Y, f(x_0))$. In fact, if
$$
\delta\, :\, \pi_1(Y, f(x_0))\, \longrightarrow\, \text{GL}({\mathfrak g})
$$
is a homomorphism corresponding to the Higgs vector bundle $(E'\, , \lambda\cdot\theta')$,
then
\begin{itemize}
\item the homomorphism $A\circ\rho^\lambda$ factors through the quotient 
$\pi_1(Y, f(x_0))$, and

\item the homomorphism $\pi_1(Y, f(x_0))\, \longrightarrow\,
\text{GL}({\mathfrak g})$ resulting from $A\circ\rho^\lambda$ differs from
$\delta$ by an inner automorphism of $\text{GL}({\mathfrak g})$.
\end{itemize}
This completes the proof.
\end{proof}

Fix a maximal compact subgroup
$$
K\, \subset\, G\, .
$$
Define
$$
{\mathcal R }^f(\pi_1(X, x_0),\, G)\,=\, \{\rho\, \in\, \text{Hom}(\pi_1(X,x_0),\, G)\,
\mid\, A\circ\rho \ \text{ factors through }~ f_*\}\, ,
$$
$$
{\mathcal R }^f(\pi_1(X, x_0),\, K)\,=\, \{\rho\, \in\, \text{Hom}(\pi_1(X,x_0),\, K)\,
\mid\, A\circ\rho \ \text{ factors through }~ f_*\}\, .
$$
Since $\pi_1(X, x_0)$ is a finitely presented group, the affine algebraic structure of $G$
produces an affine algebraic structure on ${\mathcal R }^f(\pi_1(X, x_0),\, G)$. The group $G$
acts on ${\mathcal R }^f(\pi_1(X, x_0),\, G)$ via the conjugation action of $G$ on itself. Let
$$
{\mathcal R }^f(\pi_1(X, x_0),\, G)/\!\!/G
$$
be the corresponding geometric invariant theoretic quotient. We note that this
geometric invariant theoretic quotient
${\mathcal R }^f(\pi_1(X, x_0),\, G)/\!\!/G$ is a complex affine algebraic variety. Let
$$
{\mathcal R }^f(\pi_1(X, x_0),\, K)/K
$$
be the quotient of ${\mathcal R }^f(\pi_1(X, x_0),\, K)$ for the adjoint action of $K$ on
itself.

The inclusion of $K$ in $G$ produces an inclusion of ${\mathcal R }^f(\pi_1(X, x_0),\, K)$ in
${\mathcal R }^f(\pi_1(X, x_0),\, G)$, which, in turn, gives an inclusion
\begin{equation}\label{incl}
{\mathcal R }^f(\pi_1(X, x_0),\, K)/K\, \hookrightarrow\, {\mathcal R }^f(\pi_1(X, x_0),\, G)/\!\!/G\, .
\end{equation}
Instead of working with the Zariski topology on ${\mathcal R }^f(\pi_1(X, x_0),\, G)/\!\!/G$, we consider on it the 
Euclidean topology which is induced from an embedding of this space in a complex affine space.
Indeed, such an embedding can always be obtained by considering a finite set of generators of the algebra 
of $G$-invariant regular functions on ${\mathcal R }^f(\pi_1(X, x_0),\, G)$. Moreover, this topology is
independent of the choice of such embedding, and compatible with the inclusion \eqref{incl}.

\begin{theorem}\label{thm1}
The topological space ${\mathcal R }^f(\pi_1(X, x_0),\, G)/\!\!/G$ admits a deformation
retraction to the above subset ${\mathcal R }^f(\pi_1(X, x_0),\, K)/K$.
\end{theorem}

\begin{proof}
Two elements of $\text{Hom}(\pi_1(X,x_0),\, G)$ are called equivalent if they differ by an inner
automorphism of $G$. Points of ${\mathcal R }^f(\pi_1(X, x_0),\, G)/\!\!/G$ correspond to
the equivalence
classes of homomorphisms $\rho\, \in\, \text{Hom}(\pi_1(X,x_0),\, G)$ such that the action
of $\pi_1(X,x_0)$ on $\mathfrak g$ given by $A\circ\rho$ is completely reducible, meaning that
$\mathfrak g$ is a direct sum of irreducible $\pi_1(X,x_0)$--modules. Let $(E_\rho\, ,
\theta_\rho)$ be the semistable $G$--Higgs bundle corresponding to the
above homomorphism $\rho$, and let $(\text{ad}(E_\rho)\, ,
\text{ad}(\theta_\rho))$ be the semistable adjoint
Higgs vector bundle associated to $(E_\rho\, ,\theta_\rho)$.
The above condition that the action of $\pi_1(X,x_0)$ on $\mathfrak g$ given by $A\circ\rho$ is
completely reducible is equivalent to the condition that the semistable Higgs vector bundle
$(\text{ad}(E_\rho)\, , \text{ad}(\theta_\rho))$ is polystable.

Let
$$
\phi\, :\, ({\mathcal R }^f(\pi_1(X, x_0),\, G)/\!\!/G)\times [0\, ,1]\, \longrightarrow\,
{\mathcal R }^f(\pi_1(X, x_0),\, G)/\!\!/G
$$
be the map defined by $(\rho\, ,\lambda)\, \longmapsto\, \rho^{1-\lambda}$ (defined in
\eqref{e5}), where $\rho\,\in\, \text{Hom}(\pi_1(X,x_0),\, G)$ satisfies the
condition that the action of $\pi_1(X,x_0)$ on $\mathfrak g$ given by
$A\circ\rho$ is completely reducible. It is easy to see that $\phi$ is well-defined. We note
that the point in the geometric invariant theoretic quotient
${\mathcal R }^f(\pi_1(X, x_0),\, G)/\!\!/G$ given by $\rho$ lies in the
subset ${\mathcal R }^f(\pi_1(X, x_0),\, K)/K$ if and only if the Higgs field $\theta_\rho$ on
the principal $G$--bundle $E_\rho$ vanishes identically (as before, $(E_\rho\, ,\theta_\rho)$ is
the Higgs $G$--bundle corresponding to $\rho$).

The following are straightforward to check:
\begin{itemize}
\item $\phi(z\, ,0)\, =\, z$ for all $z\, \in\, {\mathcal R }^f(\pi_1(X, x_0),\, G)/\!\!/G$,

\item $\phi(z\, ,1)\, \in \, {\mathcal R }^f(\pi_1(X, x_0),\, K)/K$ for all $z\, \in\,
{\mathcal R }^f(\pi_1(X, x_0),\, G)/\!\!/G$, and

\item $\phi(z\, ,\lambda)\, =\, z$ for all $z\, \in\,{\mathcal R }^f(\pi_1(X, x_0),\, K)/K$
and $\lambda\, \in\, [0\, ,1]$.
\end{itemize}
Therefore, the above map $\phi$ produces
a deformation retraction of ${\mathcal R }^f(\pi_1(X, x_0),\, G)/\!\!/G$
to ${\mathcal R }^f(\pi_1(X, x_0),\, K)/K$.
\end{proof}

\begin{remark}
Lemma \ref{lem1} and Theorem \ref{thm1} are also valid for morphisms 
$f\, :\, X\, \longrightarrow\, Y$ in the category of compact K\"ahler manifolds, 
under the same assumptions on $Y$ and $f_*$. The proofs of these results are 
analogous, by replacing semistability with the notion of \emph{pseudostability} (see 
\cite{BG}, \cite{BF2}).
\end{remark}

\section{Deformation retraction of the space of almost commuting elements}

Again, let $G$ be a connected complex reductive group, and $K$ be a maximal compact subgroup.
Let $$Z_G\, \subset\, G$$ be the center of $G$ and let
$$
PG\,:=\, G/Z_G
$$
be the quotient group. We note that the center of $PG$ is trivial. Let
\begin{equation}\label{q}
q\,:\, G\, \longrightarrow\, PG
\end{equation}
be the quotient map. The image
$$
PK\, :=\, q(K)\, \subset\, PG
$$
is a maximal compact subgroup of $PG$. We have $q^{-1}(PK)\,=\, K$.

Fix a positive integer $n$. Define
$$
{\rm AC}^n(G)\,=\, \{(g_1\, ,\cdots\, , g_n)\,\in\, G^n\, \mid\, g_ig_jg^{-1}_ig^{-1}_j
\,\in\, Z_G~\, \ \forall \ i\, ,j\}\, .
$$
It is a subscheme of the affine variety $G^n$. The group $G$
acts on ${\rm AC}^n(G)$ as simultaneous conjugation of the $n$ factors. Let
$$
{\rm ACE}^n(G)\,:=\, {\rm AC}^n(G)/\!\!/ G
$$
be the geometric invariant theoretic quotient. Also, define
$$
{\rm AC}^n(K)\,=\, \{(g_1\, ,\cdots\, , g_n)\,\in\, K^n\, \mid\, g_ig_jg^{-1}_ig^{-1}_j
\,\in\, Z_G~\, \ \forall \ i\, ,j\}\, .
$$
So ${\rm AC}^n(K)\,=\, {\rm AC}^n(G)\bigcap K^n$. Let
$$
{\rm ACE}^n(K)\,:=\, {\rm AC}^n(K)/K
$$
be the quotient for the simultaneous conjugation action of $K$ on the
$n$ factors. Note that the inclusion of $K$ in $G$ produces an inclusion
$${\rm ACE}^n(K)\,\hookrightarrow\, {\rm ACE}^n(G)\, .$$

\begin{proposition}\label{prop1}
Let $G$ be semisimple. Then, the topological space ${\rm ACE}^n(G)$ admits a deformation retraction to
the above subset ${\rm ACE}^n(K)$.
\end{proposition}

\begin{proof}
When $G$ is semisimple, $Z_G$ is a finite subgroup of $G$, so that the map \eqref{q} is a Galois covering.
Also, $Z_G\, \subset\, K$. 
Define ${\rm AC}^n(PG)$ and ${\rm ACE}^n(PG)$ by substituting $PG$ in place of $G$
in the above constructions. Note that ${\rm AC}^n(PG)$ parametrizes commuting
$n$ elements of $PG$ because the center of $PG$ is trivial. Similarly, define ${\rm AC}^n
(PK)$ and ${\rm ACE}^n(PK)$ by substituting $PK$ in place of $K$. So
${\rm AC}^n(PK)$ parametrizes commuting $n$ elements of $PK$. The projection
\begin{equation}\label{b2}
\beta\, :\, {\rm ACE}^n(G)\, \longrightarrow\, {\rm ACE}^n(PG)
\end{equation}
constructed using the the projection $q$ in \eqref{q} is a Galois covering with Galois group
$Z^n_G$. However it should be mentioned that ${\rm ACE}^n(G)$ need not be connected. Let
$$
\gamma\, :\, {\rm ACE}^n(K)\, \longrightarrow\, {\rm ACE}^n(PK)
$$
be the projection constructed similarly using $q$. Clearly, $\gamma$ coincides with the
restriction of $\beta$ to ${\rm ACE}^n(K)\, \subset\, {\rm ACE}^n(G)$.

There is a deformation retraction of ${\rm ACE}^n(PG)$ to ${\rm ACE}^n(PK)$
$$
\varphi\, :\, {\rm ACE}^n(PG)\times [0\, ,1]\, \longrightarrow\, {\rm ACE}^n(PG)
$$
\cite[Theorem 1.1]{FL} (see also \cite{BF1}). In particular, $\varphi\vert_{{\rm ACE}^n(PG)\times\{0\}}$ is the identity
map of ${\rm ACE}^n(PG)$.

Applying the homotopy lifting property to the covering $\beta$ in \eqref{b2}, there is a unique map
$$
\widetilde{\varphi}\, :\, {\rm ACE}^n(G)\times [0\, ,1]\, \longrightarrow\, {\rm ACE}^n(G)
$$
such that
\begin{enumerate}
\item $\beta\circ\widetilde{\varphi}\,=\, \varphi\circ (\beta\times\text{Id}_{[0,1]})$,
and

\item $\widetilde{\varphi}\vert_{{\rm ACE}^n(G)\times\{0\}}$ is the identity
map of ${\rm ACE}^n(G)$.
\end{enumerate}
This map $\widetilde{\varphi}$ is a deformation retraction of
${\rm ACE}^n(G)$ to ${\rm ACE}^n(K)$, because $\varphi$ is a deformation retraction.
\end{proof}

Proposition \ref{prop1} remains valid in the more general situation when $G$ is reductive.

\begin{theorem}
Let $G$ be a connected reductive affine
algebraic group over $\mathbb C$. Then, ${\rm ACE}^n(G)$ admits a deformation retraction to the subset ${\rm ACE}^n(K)$.
\end{theorem}

\begin{proof}
First, note that Proposition \ref{prop1} is clearly valid if $G$ is a product of copies of the multiplicative group
${\mathbb C}^*$. Hence it remains valid for any $G$ which is a product of a semisimple group
and copies of ${\mathbb C}^*$. For a general connected
reductive group $G$, consider the natural homomorphism
$$
\eta\, :\, G\, \longrightarrow\, PG \times (G/[G\, ,G])\, .
$$
It is a surjective Galois covering map, the quotient $PG \,:=\, G/Z_G$ is semisimple,
while the quotient $G/[G\, ,G]$ is a product of copies of ${\mathbb C}^*$. As mentioned
above Proposition \ref{prop1} is valid for $PG \times (G/[G\, ,G])$. Using this and
the above homomorphism $\eta$ it follows that Proposition \ref{prop1} is valid for $G$.
\end{proof}

\subsection{Deformation retraction of the space of $n$ commuting elements}

Finally, we note that the analogous result is also verified for the space of $n$ commuting elements, 
${\rm AC}^n(G)$.

\begin{theorem}
Let $G$ be a connected reductive affine algebraic group over $\mathbb C$. 
Then, the space ${\rm AC}^n(G)$ admits a deformation retraction to
the subset ${\rm AC}^n(K)$.
\end{theorem}

\begin{proof}
Since $PG$ and $PK$ have trivial center, the spaces ${\rm AC}^n(PG)$ and ${\rm AC}^n(PK)$ consist of
$n$ commuting elements: If $(g_1\, ,\cdots\, , g_n)\,\in {\rm AC}^n(PG)$, then
$$
g_i g_j = g_j g_i, \quad \text{for all} \, \, i,j \in \{1, \cdots , n \}.
$$
Therefore, it is known that ${\rm AC}^n(PG)$ admits a deformation retraction to ${\rm AC}^n(PK)$
\cite[p. 2514, Theorem 1.1]{PS}. In view of this, imitating the proof of Proposition
\ref{prop1} it follows that ${\rm AC}^n(G)$ admits a deformation retraction to
${\rm AC}^n(K)$.
\end{proof}


\end{document}